\documentclass[11pt,a4paper]{amsart}
\usepackage[T1]{fontenc}

\usepackage[initials]{amsrefs}
\usepackage{amssymb}
\usepackage{amsthm}
\usepackage{url}
\usepackage[all]{xy}

\numberwithin{equation}{section}
\numberwithin{table}{section}

\newcommand{\D}{{\mathbb D}}
\newcommand{\F}{{\mathbb F}}
\newcommand{\N}{{\mathbb N}}
\newcommand{\cG}{{\mathcal G}}
\newcommand{\cA}{{\mathcal A}}
\newcommand{\Ret}{\operatorname{Ret}}
\newcommand{\Sym}{\operatorname{Sym}}
\newcommand{\Fun}{\operatorname{Fun}}
\newcommand{\id}{\operatorname{id}}
\newcommand{\mpl}{\operatorname{mpl}}

\theoremstyle{plain}
\newtheorem{thm}{Theorem}[section]
\newtheorem{lem}[thm]{Lemma}
\newtheorem{rem}[thm]{Remark}
\newtheorem{pro}[thm]{Proposition}

\newtheorem{notation}[thm]{Notation}
\newtheorem{convention}[thm]{Convention}
\newtheorem*{conjecture*}{Conjecture}
\newtheorem{cor}[thm]{Corollary}
\newtheorem{example}[thm]{Example}

\newtheorem{defn}[thm]{Definition}

\begin{document}
\title[Extensions of set-theoretic solutions]{Extensions of
set-theoretic solutions of the Yang--Baxter equation and a conjecture of Gateva-Ivanova}
\author{L. Vendramin}

\thanks{This work was partially supported by CONICET, PICT-2014-1376 and ICTP}

\address{Depto. de Matem\'atica, FCEN,
Universidad de Buenos Aires, Pabell\'on I, Ciudad Universitaria (1428)
Buenos Aires, Argentina}
%\url{http://mate.dm.uba.ar/~lvendram/}
\email{lvendramin@dm.uba.ar}

%\date{\today}

\dedicatory{A Dino}

\begin{abstract}
	We develop a theory of extensions for involutive and nondegenerate solutions
	of the set-theoretic Yang--Baxter equation and use it to produce new families
	of solutions. As an application we construct an infinite family of
	counterexamples to a conjecture of Gateva-Ivanova related to the
	retractability of square-free solutions. 
\end{abstract}

\maketitle

\section*{Introduction}

The Yang--Baxter equation is one of the basic equations in
mathematical-physics. For that reason, in \cite{MR1183474} Drinfeld posed the
question of finding set-theoretic solutions, i.e. bijective maps 
$r\colon
X\times X\to X\times X$, where $X$ is a set, satisfying 
\[
	(r\times\id)(\id\times r)(r\times\id)=(\id\times r)(r\times\id)(\id\times r).
\]

Nondegenerate involutive solutions have received a lot of attention.  
Nondegenerate means that if one writes
\[
	r(x,y)=(\sigma_x(y),\tau_y(x)),\quad
	x,y\in X,
\]
the maps $\sigma_x$ and $\tau_x$ are bijective for each $x\in X$. Involutive
means that $r^2=\id_{X\times X}$. 

The first
results were obtained in the seminal works of Etingof,
Schedler and Soloviev \cite{MR1722951} and Gateva-Ivanova and Van den Bergh 
\cite{MR1637256}.  The theory of involutive solutions was further developed in
several papers such as \cite{MR2095675,MR2927367}, \cite{MR2132760,MR2278047},
\cite{MR2885602}, \cite{MR2584610}, \cite{MR2383056,MR2368074}, 
\cite{MR2652212,MR3177933}, \cite{MR3090121} 
and \cite{MR3190423}. 
Most of these papers focus on the so-called square-free solutions, i.e.
solutions satisfying $r(x,x)=(x,x)$ for all $x,y\in X$. Square-free solutions
have several links with other topics (semigroups of I-type, semigroups of skew
polynomial type, Bieberbach groups, quadratic algebras, etc.) and remarkable
results were obtained. However, several important questions are still open. One
of the main challenges is to construct new families of solutions. To understand
how to construct new solutions, 
Etingof, Schedler and Soloviev considered an equivalence relation  
on $X$ which induces a natural involutive solution
$\Ret(X,r)$, the so-called retraction of $(X,r)$. An
involutive solution $(X,r)$ is then called a multipermutation solution of level $m$ if
there exists a positive integer $m$ such that $\Ret^m(X,r)$ has only one
element, where the iterated retractions are defined inductively as 
$\Ret^{k}(X,r)=\Ret(\Ret^{k-1}(X,r))$ for $k>1$.

Multipermutation solutions has been intensively studied, see for example
\cite{MR2652212,MR3177933} and \cite{MR2885602}. In this context, one of the
most important unsolved problems is the following:

\begin{conjecture*}[Gateva-Ivanova]
	Every finite nondegenerate involutive square-free set-theoretic solution
	$r:X\times X\to X\times X$ such that $|X|\geq2$ is a multipermutation
	solution.
\end{conjecture*}

The conjecture was made in 2004, see \cite[2.28(I)]{MR2095675}. 

The purpose of this note is to develop a theory of extensions of involutive
nondegenerate set-theoretic solutions of the Yang--Baxter equation.  To show
the strength of the theory, we construct a set-theoretical solution that proves
that  \cite[Conjecture 2.14]{MR2776789} is not true and also answers~\cite[Open
question 6.13(3)]{MR2885602}.  Another important application of our theory is
the construction of counterexamples to the Gateva-Ivanova conjecture.  One of
these counterexamples is (up to isomorphism) the set-theoretic solution
$(X,r)$, where $X=\{1,2,\dots,8\}$ and 
\[
	r(x,y)=(\varphi_x(y),\varphi_y(x)),
	\quad
	x,y\in X,
\]
with
\begin{align*}
	&\varphi_{1}=(57), && 
	\varphi_{2}=(68), &&
	\varphi_{3}=(26)(48)(57), &&
	\varphi_{4}=(15)(37)(68),\\
	&\varphi_{5}=(13), &&
	\varphi_{6}=(24), &&
	\varphi_{7}=(13)(26)(48), &&
	\varphi_{8}=(15)(24)(37).
\end{align*}
Remarkably, extensions of this solution produce other counterexamples to the
Gateva-Ivanova conjecture. 

\medskip
The paper is organized as follows. In Section~\ref{section:preliminaries} we
review the basic notions of the set-theoretic Yang--Baxter equation. This
section also includes the basics on cycle sets, which are nonassociative
algebraic structures in bijective correspondence with nondegenerate involutive
set-theoretic solutions. In Section~\ref{section:extensions} we develop the
theory of extensions of cycle sets. Although these extensions are often hard to
compute, they provide a very powerful tool to construct new solutions. Our main
theorem appears in this section and states that if $X\to Y$ is a surjective
homomorphism of cycle sets such that all the fibers have the same cardinality
then $X$ is an extension of $Y$. In Section~\ref{section:examples} several
examples of extensions are given. Among these examples one finds the semidirect
product of cycle sets and several classes of extensions that are relatively
easy to compute. We conclude the paper with counterexamples to the
Gateva-Ivanova conjecture.

\section{Preliminaries}
\label{section:preliminaries}

Recall that a pair $(X,r)$, where $X$ is a nonempty set and 
\[
	r\colon X\times X\to X\times X,\quad
	r(x,y)=(\sigma_x(y),\tau_y(x)),\quad x,y\in X,
\]
is a bijective map, is a \emph{set-theoretic solution} of the Yang--Baxter
equation if 
\[
	r_{12}r_{23}r_{12}=r_{23}r_{12}r_{23},
\]
where $r_{12}=r\times\id$ and $r_{23}=\id\times r$. The map $r$ is
\emph{involutive} if $r^2=\id_{X\times X}$, it is \emph{nondegenerate} if the
maps $\sigma_x$ and $\tau_x$ are bijective for each $x\in X$, and it is
\emph{square-free} if $r(x,x)=(x,x)$ for all $x\in X$. 

A \emph{homomorphism} between the set-theoretic solutions $(X,r_X)$ and
$(Y,r_Y)$ is a map $f\colon X\to Y$ such that
 \[
 \xymatrix{
 X\times X
 \ar[d]_{f\times f}
 \ar[r]^-{r_X}
 & X\times X
 \ar[d]^{f\times f}
 \\
 Y\times Y
 \ar[r]^-{r_Y}
 & Y\times Y
 }
 \]
 is commutative. An \emph{isomorphism} of solutions is just a bijective
 homomorphism of solutions.

\begin{convention}
	By a \emph{solution} (of the Yang--Baxter equation) we mean a
	nondegenerate involutive set-theoretic solution of the Yang--Baxter
	equation. We consider only finite solutions.
\end{convention}

The \emph{retract relation} on a solution $(X,r)$ was first considered in
\cite{MR1722951} and it is defined by $x\sim y$ if and only if
$\sigma_x=\sigma_y$, $x,y\in X$. The \emph{retraction} of $(X,r)$ is the
solution $\Ret(X,r)=(X/{\sim},\overline{r})$ with 
\[
\overline{r}([x],[y])=([\sigma_x(y)],[\tau_y(x)]),
\]
where $[x]$ denotes the equivalence class of $x$.  A solution $(X,r)$ is called
a \emph{multipermutation solution} of level $\mpl X=m$ if there exists $m>0$ such that
$\Ret^m(X,r)$ has cardinality one, where 
\begin{align*}
	&\Ret^{k}(X,r)=\Ret(\Ret^{k-1}(X,r)),\quad k>0.
\end{align*}

The \emph{Yang--Baxter permutation group} $\cG(X,r)$ of a finite solution
$(X,r)$ is the subgroup of the symmetric group on $X$ generated by
$\{\sigma_x:x\in X\}$. 

\begin{example}
	The map 
	\[
	r(x,y)=(\sigma_x(y),\tau_y(x)),\quad 
	x,y\in\{1,2,3,4\},
	\]
	where 
	\begin{align*}
		&\sigma_1=(34), &&\sigma_2=(1324), && \sigma_3=(1423),&& \sigma_4=(12),\\
		&\tau_1=(24), && \tau_2=(1432), &&  \tau_3=(1234),&& \tau_4=(13),
	\end{align*}
	is a solution of the Yang--Baxter equation. It is not square-free and it is
	not a multipermutation solution. Moreover, 
	$\cG(X,r)\simeq\D_4$, where $\D_4$ is the dihedral group of eight elements.
\end{example}

To study solutions of the Yang--Baxter equation,  Rump introduced what is known
as cycle sets \cite{MR2132760}.  A \emph{cycle set} is a set $X$ with a binary
operation $X\times X\to X$, $(x,y)\mapsto x\cdot y$, such that the maps
$\varphi_x\colon y\mapsto x\cdot y$ are bijective for each $x\in X$ and
\[
		(x\cdot y)\cdot (x\cdot z)=(y\cdot x)\cdot (y\cdot z)
\]
for all $x,y,z\in X$.  
To describe a finite cycle set $X$, without loss of
generality we may assume that $X=\{1,2,\dots,n\}$ for some $n\in\N$ and then
write the sequence of permutations $\varphi_1,\varphi_2,\dots,\varphi_n$.

Cycle sets are in bijective correspondence with involutive nondegenerate
solutions of the Yang--Baxter equation. The correspondence is given by
\begin{equation}
	\label{eq:correspondence}
	r(x,y)=((y*x)\cdot y,y*x),\quad x,y\in X,
\end{equation}
where $y*x=z$ if and only if $x=y\cdot z$, see \cite[Prop. 1]{MR2132760}.
From the correspondence~\eqref{eq:correspondence} one obtains immediately that $r$ is square-free
if and only if $x\cdot x=x$ for all $x\in X$. Thus we say that a cycle set $X$ is \emph{square-free} 
if $x\cdot x=x$ for all $x\in X$.

A \emph{homomorphism} between the cycle sets $X$ and $Z$ is a map $f\colon
X\to Z$ such that $f(x\cdot y)=f(x)\cdot f(y)$ for all $x,y\in X$. As usual, an
\emph{isomorphism} of cycle sets is just a bijective homomorphism of cycle sets.

\begin{example}
	The square-free solution over $X=\{1,2,3\}$ given by the map
	$r\colon X\times X\to X\times X$, where  
	\begin{align*} 
		&r(1,1)=(1,1), && r(1,2)=(2,1), && r(1,3)=(3,2),\\
		&r(2,1)=(1,2), && r(2,2)=(2,2), && r(2,3)=(3,1),\\
		&r(3,1)=(2,3), && r(3,2)=(1,3), && r(3,3)=(3,3),
	\end{align*}
	has permutations
	\[
		\sigma_1=\sigma_2=\tau_1=\tau_2=\id,\quad
		\sigma_3=\tau_3=(12). 
	\]
	This solution corresponds to the square-free cycle set $X=\{1,2,3\}$ given by
	the permutations $\varphi_1=\varphi_2=\id$ and $\varphi_3=(12)$. 
\end{example}

\begin{notation} 
	For a nonempty set $S$ we write $\Fun(S,S)$ to denote set of maps $S\to S$
	and $\Sym(S)$ to denote the set of permutations on $S$. We usually write 
	$f(-)$ to denote the map $t\mapsto f(t)$.  We write $|X|$ to denote the cardinality
	of the set $X$.  
\end{notation}

\section{Dynamical extensions}
\label{section:extensions}

Inspired by the theory of extensions for racks and quandles
\cite[\S2]{MR1994219}, we develop a theory of dynamical extensions of finite
cycle sets. 

\begin{lem}
	\label{lem:dynamical}
	Let $X$ be a finite cycle set, $S$ be a finite nonempty set, and
	$\alpha\colon X\times X\times S\to\Fun(S,S)$,
	$(x,y,s)\mapsto\alpha_{x,y}(s,-)$, be a map.  Then $S\times X$ is a cycle set
	with respect to 
	\begin{equation}
		\label{eq:extension}
		(s,x)\cdot (t,y)=(\alpha_{x,y}(s,t),x\cdot y),\quad
   	x,y\in X,\;s,t\in S, 
	\end{equation}
	if and only if the maps $t\mapsto \alpha_{x,y}(s,t)$ are bijective for each 
	$x,y\in X$ and $s\in S$, and 
	\begin{equation}
		\label{eq:dynamical}
		\alpha_{x\cdot y,x\cdot z}(\alpha_{x,y}(r,s),\alpha_{x,z}(r,t))
		=\alpha_{y\cdot x,y\cdot z}(\alpha_{y,x}(s,r),\alpha_{y,z}(s,t))
	\end{equation}
	holds for all $x,y,z\in X$ and $r,s,t\in S$. 
\end{lem}

\begin{proof}
	It is straightforward to check that the operation~\eqref{eq:extension} is
	invertible if and only if the maps $t\mapsto\alpha_{x,y}(s,t)$ are invertible
	for each $x,y\in X$ and $s\in S$.  Then 
	\[
	(r,x)\cdot (s,y)=(t,z)\iff
	(s,y)=(\alpha^{-1}_{x,x*z}(r,t),x*z),
	\]
	where $*$ denotes the inverse of the binary operation of the cycle set $X$. 

	Now for $(r,x),(s,y),(t,z)\in S\times X$ one computes
	\begin{align*}
		((r,x)\cdot (s,y))&\cdot ((r,x)\cdot (t,z))\\
		&= (\alpha_{x,y}(r,s),x\cdot y)\cdot (\alpha_{x,z}(r,t),x\cdot z)\\
		&=(\alpha_{x\cdot y,x\cdot z}(\alpha_{x,y}(r,s),\alpha_{x,z}(r,t)),(x\cdot y)\cdot (x\cdot z))
	\end{align*}
	and similarly
	\begin{multline*}
		((s,y)\cdot (r,x))\cdot ((s,y)\cdot (t,z))
		=(\alpha_{y\cdot x,y\cdot z}(\alpha_{y,x}(s,r),\alpha_{y,z}(s,t)),(y\cdot x)\cdot (y\cdot z)).
	\end{multline*}
	Then the lemma follows.
\end{proof}

\begin{defn}
	Let $X$ be a finite cycle set and $S$ a nonempty finite set.  A map
	$\alpha\colon X\times X\times S\to\Sym(S)$ satisfying~\eqref{eq:dynamical} is
	called a \emph{dynamical cocycle} of $X$ with values on $S$.  
\end{defn}

\begin{notation}
	We write $Z^2(X,S)$ to denote the set of dynamical cocycles of the finite
	cycle set $X$ with values in the finite set $S$, i.e.  
	\[
	Z^2(X,S)=\{\alpha\colon X\times X\times S\to\Fun(S,S):\alpha\text{ is a dynamical cocycle}\}.
	\]
\end{notation}

\begin{defn}
	Let $X$ be a finite cycle set, $S$ a nonempty finite set and $\alpha\in Z^2(X,S)$.
	The cycle set $S\times_\alpha X$ constructed by Lemma~\ref{lem:dynamical}
	is called a \emph{dynamical extension} of $X$ by $\alpha$. 
\end{defn}

\begin{example}
	Let $X$ be a finite cycle set $X$ and $S$ a finite nonempty set. The map
	$\alpha\colon X\times X\to\Sym(S)$ given by $\alpha_{x,y}(s,t)=t$ for all
	$x,y\in X$ and $s,t\in S$ is a dynamical cocycle of $X$. This is the
	\emph{trivial dynamical cocycle} of $X$. 
\end{example}

\begin{rem}
	\label{rem:dynamical} 
	If $X$ is a finite square-free cycle set, $S$ is a nonempty finite set and
	$\alpha\in Z^2(X,S)$, then $S\times_\alpha X$ is square-free if and only if 
	\begin{equation*}
		\begin{aligned}
		&\alpha_{x,x}(s,s)=s\quad
		\text{for all $x\in X$ and $s\in S$.}
		\end{aligned}
	\end{equation*}
\end{rem}

\begin{defn}
	We say that a homomorphism $p\colon X\to Y$ between the finite cycle sets $X$
	and $Y$ is a \emph{covering map} if it is surjective and all the fibers
	$p^{-1}(y)$, where $y\in Y$, have the same cardinality. 
\end{defn}

\begin{defn}
	A covering map $X\to Y$ is \emph{trivial} if either $|Y|=1$ or $|Y|=|X|$. 
\end{defn}

\begin{defn}
	A finite cycle set $X$ is \emph{simple} if $|X|>1$ and any covering map
	$X\to Y$ is trivial. 
\end{defn}

\begin{example}
	Every cycle set with a prime number of elements is simple. 
\end{example}

\begin{example}
	Let $X=\{1,2,3,4\}$ be the cycle set given by 
	\[
	\varphi_1=(14),\quad
	\varphi_2=(1342),\quad
	\varphi_3=(23),\quad
	\varphi_4=(1243).
	\]
	Let us prove that $X$ is simple. If $p\colon X\to Y$ is a covering map,
	then $|Y|=2$. Thus $Y=\{a,b\}$ with $a\ne b$. Since there
	are only two cycle sets with two elements, there are two cases to consider.  

	Suppose first that the cycle set structure over $Y$ is given by
	$\psi_a=\psi_b=\id$. Since $p$ is a cycle set homomorphism and
	$\psi_a=\psi_b=\id$, it follows that $p(x\cdot x)=p(x)\cdot p(x)=p(x)$ for
	all $x\in X$. But $x\cdot x\ne x$ for all $x\in X$ and hence  $p(x)=p(y)$
	for all $x,y\in X$, a contradiction.

	Now suppose that $\psi_a=\psi_b=(ab)$ and $p(1)=a$. From $1\cdot 1=4$ one
	obtains that $p(4)=b$ and hence $p(2)=p(4\cdot 1)=p(4)\cdot p(1)=b\cdot
	a=b$.  Then $a=a\cdot b=p(1)\cdot p(2)=p(1\cdot 2)=p(2)=b$, a
	contradiction.
\end{example}

\begin{thm}
	\label{thm:dynamical}
	Let $X$ be a finite cycle set and assume that $X$ admits a covering map
	$p:X\to Y$ onto a finite cycle set $Y$.  Then there exists 
	a finite nonempty set $S$ and a dynamical cocycle $\alpha\in Z^2(X,S)$
	such that $X$ is isomorphic to the dynamical extension $S\times_\alpha Y$.
\end{thm}

\begin{proof}
	Since all the fibers $p^{-1}(y)$ have the same
	cardinality, there is a finite
	nonempty set $S$ and there are bijections $f_y\colon p^{-1}(y)\to S$ for
	each $y\in Y$. Let $\alpha\colon Y\times Y\times S\to\Fun(S,S)$ be the map
	given by
	\[
		\alpha_{y,z}(s,t)=f_{y\cdot z}(f^{-1}_y(s)\cdot f^{-1}_{z}(t)),\quad
		s,t\in S,\;y,z\in Y. 
	\]
	Then $\alpha\in Z^2(Y,S)$ and hence  
	$S\times Y$ is a cycle set with respect to 
	\[
		(s,y)\cdot (t,z)=(\alpha_{y,z}(s,t),y\cdot z),\quad
		s,t\in S,\;y,z\in Y. 
	\]
	The map $\phi\colon X\to S\times Y$ given by $x\mapsto
	(f_{p(x)}(x),p(x))$ is a homomorphism of cycle sets. 
	Furthermore, $\phi$ is
	bijective since the map $\psi\colon S\times Y\to X$ given
	by $\psi(s,y)=f^{-1}_{y}(s)$ satisfy $\phi\circ\psi=\id_{S\times Y}$
	and $\psi\circ\phi=\id_X$. 
\end{proof}

\begin{cor}
	Let $X$ be a finite nonsimple cycle set.  Then there exists a finite cycle
	set $Y$ with $1<|Y|<|X|$, a finite nonempty set $S$ and a dynamical
	cocycle $\alpha\in Z^2(X,S)$ such that $X$ is isomorphic to $S\times_\alpha
	Y$.
\end{cor}

\begin{proof}
	If follows immediately from Theorem~\ref{thm:dynamical}.
\end{proof}

\begin{example}
	Let $X=\{1,\dots,6\}$ be the cycle set given by
	\begin{align*}
		\varphi_1=\varphi_2=(12)(34)(56),&&
		\varphi_3=\varphi_5=(12)(3654),&&
		\varphi_4=\varphi_6=(12)(3456),
	\end{align*}
	and let $Y=\{1,2\}$ be the cycle set given by
	\[
		\psi_1=\psi_2=(12).
	\]
	The map 
	\[
	p\colon X\to Y,\quad
	p(x)=\begin{cases}
		1 & \text{if $x$ is odd},\\
		2 & \text{otherwise},
	\end{cases}
	\]
	is a surjective cycle set homomorphism. Since the fibers $p^{-1}(1)$ and $p^{-1}(2)$
	have both three elements, Theorem~\ref{thm:dynamical} implies that there
	exists $S=\{a,b,c\}$ and $\alpha\in Z^2(X,S)$ such
	that $X\simeq S\times_\alpha Y$. The bijections
	\[
		f_1\colon 1\mapsto a,\;3\mapsto b,\;5\mapsto c,
		\quad
		f_2\colon 2\mapsto a,\;4\mapsto b,\;6\mapsto c,
	\]
	and a direct calculations yield
	\[
		\alpha_{x,y}(s,-)=\begin{cases}
			(bc) & \text{if $(x,y,s)\in\{(1,1,b),(1,1,c),(2,2,b),(2,2,c)\}$},\\
			\id & \text{otherwise}.
		\end{cases}
	\]
\end{example}

\begin{defn}
	Two dynamical cocycles $\alpha,\beta\in Z^2(X,S)$ are \emph{cohomologous} if
	there exists a map $\gamma\colon X\to\Sym(S)$, $x\mapsto\gamma_x$, such that
	\begin{equation}
		\label{eq:cohomologous}
		\gamma_{x\cdot y}\left(\alpha_{x,y}(\gamma^{-1}_x(s),\gamma^{-1}_{y}(t))\right)
		=\beta_{x,y}(s,t)
	\end{equation}
	for all $x,y\in X$ and $s,t\in S$. 
\end{defn}

\begin{pro}
	Let $X$ be a finite cycle set, $S$ a nonempty finite set and
	$\alpha,\beta\in Z^2(X,S)$. The following hold:
	\begin{enumerate}
		\item If $\alpha$ and $\beta$ are cohomologous by $\gamma$, then
			\begin{equation*}
				F\colon S\times_\alpha X\to S\times_\beta X,\quad
				(s,x)\mapsto (\gamma_x(s),x),
			\end{equation*}
			is a bijective cycle set homomorphism.
		\item If there is a bijective cycle set homomorphism $F\colon
			S\times_\alpha X\to S\times_\beta X$ such that $p_\beta\circ F=p_\alpha$,
			where $p_\alpha\colon S\times_\alpha X\to X$ and $p_\beta\colon
			S\times_\beta X\to X$ are the canonical surjections, then $\alpha$ and
			$\beta$ are cohomologous.
	\end{enumerate}
\end{pro}

\begin{proof}
	We first prove (1). Since $\gamma_x\in\Sym(S)$ for all $x\in X$, the map $F$
	is bijective.  Let us prove that $F$ is a cycle set homomorphism:  For 
	$(s,x),(t,y)\in S\times X$ one computes 
	\begin{align}
		\label{eq:cohomologous:1}
		F((s,x)\cdot (t,y))
		=F(\alpha_{x,y}(s,t),x\cdot y)
		=(\gamma_{x\cdot y}(\alpha_{x,y}(s,t)),x\cdot y)
	\end{align}
	and similarly 
	\begin{align}
		\label{eq:cohomologous:2}
		F(s,x)\cdot F(t,y)
		=(\gamma_x(s),x)\cdot (\gamma_y(t),y)
		=(\beta_{x,y}(\gamma_x(s),\gamma_y(t)),x\cdot y).
	\end{align}
	Thus the claim follows from Equation~\eqref{eq:cohomologous}. 
	
	Now we prove (2). Since $F$ is bijective and $p_\beta\circ F=p_\alpha$, we
	may assume that $F(s,x)=(f(s,x),x)$ for some $f\colon S\times X\to S$. Then
	the maps $\gamma_x\colon S\to S$, $s\mapsto f(s,x)$, are bijective for each
	$x\in X$. Since $F$ is a cycle set homomorphism, the claim follows 
	from Equations~\eqref{eq:cohomologous:1} and~\eqref{eq:cohomologous:2}.
\end{proof}

\section{Examples}
\label{section:examples}

In this section we collect some examples of (dynamical) cocycles and
(dynamical) extensions of cycle sets. 

\begin{example}
	[Rump's semidirect product of cycle sets] 
	Let $X$ and $S$ be finite cycle sets and suppose that $X$ acts on $S$,
	which means that there is a map $X\times S\to S$, $(x,s)\mapsto xs$, such
	that 
	\begin{enumerate}
		\item $x(s\cdot t)=xs\cdot xt$ for all $x\in X$ and $s,t\in S$, 
		\item $(x\cdot y)xs=(y\cdot x)ys$ for all $x,y\in X$ and $s\in S$, and 
		\item The map $S\to S$, $s\mapsto xs$, is bijective for all $x\in X$.
	\end{enumerate}
	By \cite[Thm. 1]{MR2442072}, $S\times X$ is a cycle
	set with
	\[
		(s,x)\cdot (t,y)=((x\cdot y)s\cdot (y\cdot x)t,x\cdot y),\quad
		(s,x),(t,y)\in S\times X.
	\]
	From Lemma~\ref{lem:dynamical} one obtains that the map $\alpha\colon X\times
	X\times S\to\Fun(S,S)$ given by $\alpha_{x,y}(s,t)=(x\cdot y)s\cdot (y\cdot
	x)t$ is then a dynamical cocycle over $X$.
\end{example}

Let $\mathbf{n}_m$ be defined as the minimal integer so that there exists a
square-free multipermutation solution $X$ of size $\mathbf{n}_m$ and $\mpl
X=m$. In~\cite[Open question 6.13(3)]{MR2885602} it was asked if 
$\mathbf{n_m}=2^{m-1}+1$ holds for all $m\in\N$. 

\begin{example}
	\label{exa:GI}
        Let $Y=\{1,2,3\}$ be the cycle set given by $\varphi_1=\varphi_2=\id$,
        $\varphi_3=(12)$ and $S=\{a,b\}$ be a set with two elements. Write
        \[
                \cA=\{(1,2,a),(1,2,b),(1,3,b),(2,1,a),(2,1,b),(2,3,b)\}\subseteq X\times X\times S.
        \]
        The map
        $\alpha\colon Y\times Y\times S\to\Sym(S)$ given by
        \[
        \alpha_{x,y}(s,-)=\begin{cases}
                (ab) & \text{if $(x,y,s)\in\cA$},\\
                \id & \text{otherwise},
        \end{cases}
        \]
        is a dynamical cocycle of $Y$. Let 
        \[
                x_j=\begin{cases}
                        (a,j) & \text{if $j\in\{1,2,3\}$,}\\
                        (b,j-3) & \text{if $j\in\{4,5,6\}$.}
                \end{cases}
        \]
        The extension $X=S\times_\alpha Y$ is the cycle set $\{x_1,\dots,x_6\}$ given by
        \begin{gather*}
			\psi_{x_1}=(x_2x_5),\quad
			\psi_{x_2}=(x_1x_4),\quad
			\psi_{x_3}=(x_1x_2)(x_4x_5),\\
			\psi_{x_4}=(x_2x_5)(x_3x_6),\quad
			\psi_{x_5}=(x_1x_4)(x_3x_6),\quad
			\psi_{x_6}=(x_1x_2)(x_4x_5).
        \end{gather*}
		This cycle set corresponds to a square-free multipermutation solution of level
		four. This solution satisfies $4=\mpl X>\log_26$ and hence
		\cite[Conjecture 2.14]{MR2776789} is not true. This
		example also shows that $\mathbf{n}_4\leq 6<2^{3}+1$ and answers \cite[Open
		question 6.13(3)]{MR2885602}. 
\end{example}

\subsection{Constant cocycles}
A very interesting class of dynamical cocycles consists of the so-called
\emph{constant cocycles}. By definition, a
constant cocycle $\alpha$ can be written as $\alpha_{x,y}(s,t)=\beta_{x,y}(t)$
for some $\beta\colon X\times X\to\Sym(S)$. For constant cocycles the cocycle
condition~\eqref{eq:dynamical} is 
\[
	\beta_{x\cdot y,x\cdot z}\beta_{x,z}=\beta_{y\cdot x,y\cdot z}\beta_{y,z}\quad
\]
for all $x,y,z\in X$.

\begin{example}
	Let $X=\{1,2,3\}$ be the square-free cycle set given by 
	\[
	\varphi_1=\varphi_2=\id,\quad\varphi_3=(12),
	\]
	and let $S=\{a,b\}$ be a set with two elements.  
	The map $\beta\colon X\times X\to\Sym(S)$ given by
	\[
		\beta_{1,1}=\beta_{1,3}=\beta_{2,2}=\beta_{2,3}=\beta_{3,1}=\id,\quad
		\beta_{1,2}=\beta_{2,1}=\beta_{3,2}=\beta_{3,3}=(ab).
	\]
	is a constant cocycle of $X$. To compute the extension $S\times_{\beta}X$,
	let 
	\[
		x_j=\begin{cases}
			(a,j) & \text{if $j\in\{1,2,3\}$},\\
			(b,j-3) & \text{if $j\in\{4,5,6\}$}.
		\end{cases}
	\]
	Then $S\times_{\beta} X$ is the cycle set over $\{x_1,\dots,x_6\}$ given by
	\begin{gather*}
		\psi_{x_1}=\psi_{x_4}=(x_2x_5),\\
		\psi_{x_2}=\psi_{x_5}=(x_1x_4),\\
		\psi_{x_3}=\psi_{x_6}=(x_1x_2x_4x_5)(x_3x_6).
	\end{gather*}
\end{example}

\begin{example}
	\label{exa:CES}
	For a finite abelian group $A$ (written additively) we consider
	constant cocycles $\beta\colon X\times X\to\Sym(A)$ of the form
	\[
	\beta_{x,y}(b)=b+f(x,y),
	\]
	where $f\colon X\times X\to A$ is a map. By Lemma~\ref{lem:dynamical}, 
	$A\times X$ with 
	\[
	(a,x)\cdot (y,b)=(b+f(x,y),x\cdot y),\quad
	a,b\in A,\;x,y\in X,
	\]
	is a cycle set if and only if 
	\begin{equation}
		\label{eq:cocycle}
		f(x,z)+f(x\cdot y,x\cdot z)=f(y,z)+f(y\cdot x,y\cdot z)
	\end{equation} 
	for all $x,y,z\in X$. Such a map $f$ will be called a \emph{cocycle} of $X$ and it will be used
	to construct a cycle set structure over $A\times X$.
\end{example}

\begin{rem}
	The cocycles of Example~\ref{exa:CES} are related to the homology theory for
	set-theoretic solutions of the Yang--Baxter equation of Carter, Elhamdadi and
	Saito. This homology is useful to construct invariants of virtual knots, see
	for example~\cite[\S2]{MR2128041} and \cite{MR2515811}.
\end{rem}

Cocycles like those of Example~\ref{exa:CES} are often easy to compute. To
compute all the cocycles of a cycle set $X$ with values in, say, the finite
field with $q$ elements $A=\F_q$, one needs to solve a linear system (over
$\F_q$) with $|X|^3$ equations and $|X|^2$ unknowns. 

\begin{example}
	Let us write $\F_2=\{0,1\}$. Let 
	$X=\{1,2,3\}$ be the cycle set given by the permutations
	\[
	\varphi_1=\varphi_2=\id,\quad
	\varphi_3=(12),
	\]
	and $f\colon X\times
	X\to\F_2$ be the cocycle given by
	\[
		f(x,y)=\begin{cases}
			1 & \text{if $(x,y)=(3,2)$},\\
			0 & \text{otherwise}.
		\end{cases}
	\]
	For $j\in\{1,\dots,6\}$ let
	\[
		x_j=\begin{cases}
			(0,j) & \text{if $j\in\{1,2,3\}$},\\
			(1,j-3) & \text{if $j\in\{4,5,6\}$}.
		\end{cases}
	\]
	The extension of $X$ by $f$ is the cycle set structure over
	$\{x_1,\dots,x_6\}$ given by
	\[
	\psi_{x_1}=\psi_{x_2}=\psi_{x_4}=\psi_{x_5}=\id,\quad
	\psi_{x_3}=\psi_{x_6}=(x_1x_2x_4x_5).
	\]
\end{example}

\begin{example}
	Let $X=\{1,2,3,4,5\}$ be the cycle set defined by
	\[
		\varphi_1=\varphi_2=(45),\quad
		\varphi_3=(1425),\quad
		\varphi_4=\varphi_5=(12),
	\]
	and let $f\colon X\times X\to\F_2$ be the cocycle given by the matrix
	\[
	F=\begin{pmatrix}
		0 & 0 & 0 & 0 & 0\\
		0 & 0 & 0 & 0 & 0\\
		0 & 1 & 0 & 0 & 1\\
		1 & 1 & 0 & 0 & 0\\
		1 & 1 & 0 & 0 & 0\\
	\end{pmatrix}.
	\]
	(This means that $f(x,y)$ is the $(x,y)$-entry of the matrix $F$.) Let us
	compute the extension of $X$ by $f$. For $j\in\{1,\dots,10\}$ let 
	\[
	x_j=\begin{cases}
		(0,j) & \text{if $1\leq j\leq 5$},\\
		(1,j-5) & \text{if $6\leq j\leq 10$}.
	\end{cases}
	\]
	The extension of $X$ by $f$ is the cycle set over 
	$\{x_1,x_2,\dots,x_{10}\}$ given by 
	\begin{gather*}
		\psi_{x_1}=\psi_{x_2}=\psi_{x_6}=\psi_{x_7}=(x_4x_5)(x_9x_{10}),\\
		\psi_{x_3}=\psi_{x_8}=(x_1x_4x_2x_{10})(x_5x_6x_9x_7),\\
		\psi_{x_4}=\psi_{x_5}=\psi_{x_9}=\psi_{x_{10}}=(x_1x_7)(x_2x_6).
	\end{gather*}
\end{example}

\begin{example}
	Let us write $\F_3=\{0,1,2\}$. Let 
	$X=\{1,2,3,4\}$ be the cycle set given by 
	\[
	\varphi_1=\varphi_2=\varphi_3=\id,\quad
	\varphi_4=(23),
	\]
	and $f\colon X\times X\to\F_3$ be the cocycle given by
	\[
		f(x,y)=\begin{cases}
			2 & \text{if $x=y$ or $y=4$},\\
			1 & \text{if $(x,y)=(4,3)$},\\
			0 & \text{otherwise}.
		\end{cases}
	\]
	For $j\in\{1,\dots,12\}$ let 
	\[
	x_j=\begin{cases}
		(0,j) & \text{if $1\leq j\leq 4$},\\
		(1,j-4) & \text{if $5\leq j\leq 8$},\\
		(2,j-8) & \text{if $9\leq j\leq 12$}.
	\end{cases}
	\]
	The extension of $X$ by $f$ is the cycle set over
	$\{x_1,x_2,\dots,x_{12}\}$ given by 
	\begin{gather*}
        \psi_{x_1}=\psi_{x_5}=\psi_{x_9}=(x_1x_9x_5)(x_4x_{12}x_8),\\
        \psi_{x_2}=\psi_{x_6}=\psi_{x_{10}}=(x_2x_{10}x_6)(x_4x_{12}x_8),\\
        \psi_{x_3}=\psi_{x_7}=\psi_{x_{11}}=(x_3x_{11}x_7)(x_4x_{12}x_8),\\
		\psi_{x_4}=\psi_{x_8}=\psi_{x_{12}}=(x_2x_3x_6x_7x_{10}x_{11})(x_4x_{12}x_8).
	\end{gather*}
\end{example}

\subsection{Counterexamples to the Gateva-Ivanova conjecture}

We use a dynamical extension of the square-free solution over $\{1,2\}$ to
produce counterexamples to the conjecture.

\begin{example}
	\label{exa:counterexample}
	Let us consider the cycle set over the set $Y=\{1,2\}$ given by
	$\varphi_1=\varphi_2=\id$ and let $S=\{a,b,c,d\}$ be a set with four
	elements.  The map $\alpha\colon Y\times Y\times S\to\Sym(S)$ given by
	\begin{gather*}
		\alpha_{1,1}(a,-)=\alpha_{1,1}(b,-)=\alpha_{2,2}(a,-)=\alpha_{2,2}(b,-)=(cd),\\
		\alpha_{1,1}(c,-)=\alpha_{1,1}(d,-)=\alpha_{2,2}(c,-)=\alpha_{2,2}(d,-)=(ab),\\
		\alpha_{1,2}(a,-)=\alpha_{1,2}(c,-)=\alpha_{2,1}(a,-)=\alpha_{2,1}(c,-)=\id,\\
		\alpha_{1,2}(b,-)=\alpha_{1,2}(d,-)=\alpha_{2,1}(b,-)=\alpha_{2,1}(d,-)=(ac)(bd),
	\end{gather*}
	is a dynamical cocycle of $Y$.  
	For $j\in\{1,\dots,8\}$ let 
	\[
    x_j=\begin{cases}
        (a,j)    &  \text{if $1\leq j\leq 2$},\\
        (b,j-2)  &  \text{if $3\leq j\leq 4$},\\
        (c,j-4)  &  \text{if $5\leq j\leq 6$},\\
        (d,j-6)  &  \text{if $7\leq j\leq 8$}.
    \end{cases}
	%	x_j=\begin{cases}
	%		(a,j) & \text{if $1\leq j\leq 4$},\\
	%		(b,j-4) & \text{if $5\leq j\leq 8$}.\\
	%	\end{cases}
	\]
	Then $X=S\times_\alpha Y$ is the square-free cycle set over 
	$\{x_1,\dots,x_8\}$ given by 
	\begin{gather*}
		\psi_{x_1}=(x_5x_7),\\
		\psi_{x_2}=(x_6x_8),\\
		\psi_{x_3}=(x_2x_6)(x_4x_8)(x_5x_7),\\
		\psi_{x_4}=(x_1x_5)(x_3x_7)(x_6x_8),\\
		\psi_{x_5}=(x_1x_3),\\
		\psi_{x_6}=(x_2x_4),\\
		\psi_{x_7}=(x_1x_3)(x_2x_6)(x_4x_8),\\
		\psi_{x_8}=(x_1x_5)(x_2x_4)(x_3x_7).
	\end{gather*}
	Using the cycle set $X$ we construct a counterexample to 
	the Gateva-Ivanova conjecture: 
	Let
	\[
		r\colon X\times X\to X\times X,\quad
        r(x_i,x_j)=(\psi_{\psi_{x_j}^{-1}(x_i)}(x_j),\psi^{-1}_{x_j}(x_i)),\quad
		x_i,x_j\in X.
	\]
    Since $\psi_x^{-1}=\psi_x$ for all $x\in X$ and
    $\psi_{\psi_y(x)}(y)=\psi_x(y)$ for all $x,y\in X$, it follows that
    \[
    r(x_i,x_j)=(\psi_{x_i}(x_j),\psi_{x_j}(x_i)).
    \]
    The solution $(X,r)$ satisfies $\Ret(X,r)=(X,r)$ and hence 
    $(X,r)$ is a square-free
    solution which is not a multipermutation solution.
\end{example}

\begin{rem}
	In \cite[Cor. 2.9]{MR2652212} Ced\'o, Jespers and Okni{\'n}ski proved that
	every square-free solution with an abelian Yang--Baxter permutation group is
	a multipermutation solution.  The Yang--Baxter permutation group of 
	the solution $(X,r)$ of Example~\ref{exa:counterexample} is isomorphic to
	\begin{multline*}
		\langle (57), (68), (26)(48)(57), (15)(37)(68), \\(13), (24), (13)(26)(48), (15)(24)(37) 
		\rangle\simeq\D_4\times\D_4,
	\end{multline*}
	where $\D_4$ denotes the dihedral group of $8$ elements. 
\end{rem}

We conclude the paper with more counterexamples to the Gateva-Ivanova
conjecture.  These counterexamples are constructed as extensions of the
solution $(X,r)$ of Example~\ref{exa:counterexample}. The key ingredient is a
result of Ced\'o, Jespers and Okni{\'n}ski about liftings of multipermutation
solutions, see~\cite[Lemma 4]{MR3177933}.  This lemma states that if $Y\to Z$
is a surjective homomorphism of solutions and $Y$ is a multipermutation
solution then so is $Z$.

\medskip
We say that a cycle set $X$ is a \emph{multipermutation} cycle set if its
corresponding solution $(X,r_X)$ is a multipermutation solution. 

\begin{lem}
	\label{lem:CJO}
	Let $X$ be a finite square-free cycle set, let $S$ be a nonempty set and
	let $\alpha\in Z^2(X,S)$ with $\alpha_{x,x}(s,s)=s$ for all $x\in X$ and
	$s\in S$.  Assume that $X$ is not a multipermutation cycle set. Then
	$S\times_\alpha X$ is square-free and it is not a multipermutation cycle
	set.
\end{lem}

\begin{proof}
	The dynamical extension $S\times_\alpha X$ of the cycle set $X$ is a
	square-free cycle set by Lemma~\ref{lem:dynamical} and
	Remark~\ref{rem:dynamical}. It induces a surjective homomorphism of solutions
	\[
		(S\times X)\times(S\times X)\to X\times X,\quad
		((s,x),(t,y))\mapsto (x,y).
	\]
	Thus the claim follows from~\cite[Lemma 4]{MR3177933}.
\end{proof}

\subsection*{Acknowledgement}

The author wishes to thank Ferran Ced\'o for comments on the paper, Travis
Schedler for the list of nondegenerate involutive set-theoretic solutions of
size $\leq 8$ and the ICTP for the support and for the working atmosphere.  

\def\cprime{$'$}

%\bibliographystyle{abbrv}
%\bibliography{refs}
 
\end{document}